\documentclass[a4paper,10pt,reqno]{amsart}

\usepackage{amsfonts}
\usepackage{amsmath}
\usepackage{amssymb}
\usepackage{nicefrac}
\usepackage{enumitem}
\usepackage{url,bm}
\usepackage{color}      
\usepackage{capt-of}
\usepackage{graphicx}
\usepackage{subcaption}
\usepackage{hyperref}
\usepackage{verbatim} 
\allowdisplaybreaks

\newcommand{\C}{\mathbb{C}}
\newcommand{\Z}{\mathbb{Z}}

\newcommand{\mK}{\mathcal{K}}
\newcommand{\mN}{\mathcal{N}}
\newcommand{\mH}{\mathcal{H}}
\newcommand{\mW}{\mathcal{W}}
\newcommand{\mV}{\mathcal{V}}
\newcommand{\mF}{\mathcal{F}}

\newcommand{\PhiK}{\Phi_{\!_\mK}}
\newcommand{\mWK}{\mathcal{W}_{\!_\mK}}
\newcommand{\mChiK}{\chi_{\!_\mK}}

\newtheorem{theorem}{Theorem}[section]

\newtheorem{corollary}[theorem]{Corollary}
\newtheorem{proposition}[theorem]{Proposition}

\DeclareMathOperator{\Span}{span}   
\DeclareMathOperator{\supp}{supp} 
\DeclareMathOperator{\Tr}{Tr}
\DeclareMathOperator{\card}{card}

\theoremstyle{remark}
\newtheorem{remark}[theorem]{Remark}
\newtheorem{example}[theorem]{Example}

\theoremstyle{definition}
\newtheorem{definition}[theorem]{Definition}

\setlist[enumerate,1]{label={\upshape(\roman*)}}

\title{Gabor fusion frames generated by difference sets}
\author{Irena Bojarovska and Victoria Paternostro}


\begin{document}

\address{\textrm{(I. Bojarovska)}
Technische Universit\"{a}t Berlin,  Berlin, Germany;}
\email{bojarovska@math.tu-berlin.de}

\address{\textrm{(V. Paternostro)}
Departamento de Matem\'atica, Facultad de Ciencias Exactas y Naturales, Universidad  de Buenos Aires, Argentina and IMAS-CONICET, Consejo  Nacional de Investigaciones Cient\'ificas y T\'ecnicas, Argentina}
\email{vpater@dm.uba.ar}

 \begin{abstract}
Collections of time- and frequency-shifts of suitably chosen generators (Alltop  or random vectors) proved successful for many applications in sparse recovery and related fields.
It was shown in \cite{xia2005achieving} that taking a characteristic function of a difference set as a generator, and considering only the frequency shifts, gives an equaingular tight frame for the subspace they span. In this paper, we investigate the system of all $N^2$ time- and frequency-shifts of a difference set in dimension $N$ via the mutual coherence,  and compare numerically its sparse recovery effectiveness with Alltop and random generators.
We further view this Gabor system as a fusion frame, show that it is optimally sparse, and moreover an equidistant tight fusion frame, i.e. it is an optimal Grassmannian packing.
 \end{abstract}
\maketitle
\begin{section}{Introduction}

Gabor frames \cite{pfander2013gabor}, which are collections of time- and frequency- shifts (translations and modulations) of a chosen generator, have been shown useful for a variety of applications in signal processing related to signals sparse in a Gabor system, for example, model selection (also called sparsity pattern recovery) \cite{bajwa2010gabor}, and channel estimation and identification \cite{pfander2008identification}.  
A crucial property of a Gabor system is that for any unitnorm nonzero generator $v\in \C^N,$ it constitutes an $N$-tight frame \cite{lawrence2005linear, pfander2013gabor}. In the aforementioned applications, two main generators have shown to be particularly useful from both a theoretical and practical point of view: Alltop  and random vectors (see e.g. \cite{pfander2013gabor}). The theoretical guarantees come from the near optimal coherence properties of the Gabor frames generated by these vectors.

The problem of finding frames with optimal (in)coherence (in the sense of achieving the Welch bound, or equivalently being equiangular tight frame (ETF)) is of great importance for signal processing applications, as well as for other areas of mathematics. One example is coding theory, where one looks for maximum-Welch-bound-equality (MWBE) codebooks \cite{xia2005achieving}. Another example is line packing in Grassmannian manifolds, where one seeks $N$ lines in the $K$-dimensional space so that the maximum chordal distance between any two lines is minimized \cite{conway1996packing}. These equivalent problems are very difficult and analytic constructions are very limited, known to date only for certain parameters of $N$ and $K$ (see \cite{fickus2015tables} for a comprehensive overview of known results). One example of such constructions comes from combinatorial design theory \cite{dinitz1992contemporary}: 
if we take the so-called \textit{difference sets}, which are subsets of $\{ 0, 1, \ldots, N-1 \}$ with $K$ elements with certain properties (see Definition \ref{def:diff-sets}), 
and choose $K$ rows of the discrete Fourier transform matrix indexed 
according to the elements of this set -- as it was shown in \cite{xia2005achieving} -- we will obtain an ETF of $N$ vectors in dimension $K.$ For other large families of ETF inspired from design theory see \cite{fickus2012steiner, jasper2014kirkman}.

In the first part of the paper, we will view the characteristic function of a different set as an element of $\C^N$ and investigate the following question: what type of coherence properties does the full system of modulations and translations, generated by a difference set exhibit? Here the corresponding optimal packing problem is to pack $N^2$ lines in $N$-dimensional space. 
Although for some difference sets the mutual coherence can be asymptotically small, we will show that achieving the Welch bound for full Gabor frames generated by the characteristic function of difference sets is not possible. However, in the light of compressed sensing, our numerical results will show that the Gabor measurements generated by some known difference sets are suitable for recovering  sparse signals, and have a recovery rate of the order of the Gabor measurements generated by random or Alltop vectors.

Having a redundant system instead of a basis for representation of the signal does often prove helpful in order to gain stability, robustness against noise and erasures etc. The search for representations which suit the different needs in signal processing led to the development of the powerful concept of fusion frames \cite{casazza2004frames, casazza2008fusion}, which are frame-like collections of subspaces in a Hilbert space. As in the classical frame setting, in this framework one is also interested in constructing fusion frames with prescribed optimality properties, which include analogous coherence and equiangularity conditions.

Thus, in the second part of the paper, we view the aforementioned Gabor system as a fusion frame (we consider the subspaces of $\C^N$ spanned by the modulations of a fixed translation of the  generator). We will prove that having a difference set as a generator actually allows us to obtain the following desirable optimality properties: our fusion frame is tight, optimally sparse, and equidistant. 

At the end, we solve numerically the problem of recovering sparse signals from Gabor frame and Gabor fusion frame measurements, generated by difference sets. The sparsity is understood differently in each of the cases: in the first case, the signal is a linear combination of a small number of modulations and translations of the generator. For Alltop and random generators this problem was considered in \cite{bajwa2010gabor,pfander2008identification}. The second case is fusion sparsity, when the signal which lies in a union of few subspaces (in our case, involving only few translations) is measured with a measurement matrix and one would like to recover it using a mixed $\ell_2\slash \ell_1$ minimization problem \cite{ayaz2014uniform,ayaz2013sparse,boufounos2011sparse}.

The paper is organized as follows. In Section \ref{sec:gabor}, after introducing the main objects and their basic properties, we start with the study of the mutual coherence of the Gabor frame generated by a characteristic function of a difference set. We provide a formula which depends on the parameters of the difference sets and study the question of achieving the Welch bound.
In Section \ref{sec:ff} we switch to the properties of the Gabor-like fusion frame, generated by a difference set. We prove three important properties of this construction: tightness, equidistance, and optimal sparsity. 
Finally, in Section \ref{sec:num}, we provide numerical experiments to demonstrate the effectiveness of both proposed constructions in solving the sparse recovery problem. 
\end{section}

\begin{section}{Gabor systems generated by difference sets} \label{sec:gabor}

\begin{subsection}{Background}
Let us start with a formal definition of a frame. For more details on the theory of finite frames, see \cite{casazza2012finite}.

\begin{definition}
A family of nonzero  vectors $\Phi=\{ \phi_i \}_{i=1}^M$ in  $\C^N$ is called a \textit{finite frame}  for $\C^N,$ if there exist
constants $0 < A \leq B < \infty$ such that 
\begin{equation} \label{eq:deff}
 A \Vert x \Vert^2 \leq \sum_{i=1}^M \vert \langle x,\phi_i \rangle \vert^2 \leq B \Vert x \Vert^2 \quad \text{for all } x \in \C^N.
\end{equation}
\end{definition}
If $A=B$ is possible in \eqref{eq:deff}, then $\Phi=\{ \phi_i \}_{i=1}^M$  is called an $A$-\textit{tight frame}, and if additionally $A=B=1$ is possible, then it is a \textit{Parseval frame}.
If there exist a constant $c$ such that $\Vert \phi_i \Vert = c$ for all $i=1,2,\ldots,M,$ then $\Phi=\{ \phi_i \}_{i=1}^M$  is an 
\textit{equal norm frame}. If $c=1,$ $\Phi=\{ \phi_i \}_{i=1}^M$  is a \textit{unit norm frame}.
If there exists a constant $c$ such that $\vert \langle \phi_i, \phi_j \rangle \vert = c$ for all $i \neq j,$ then $\Phi=\{ \phi_i \}_{i=1}^M$  is called an \textit{equiangular frame}.

An important quantity that measures the relations between the frame elements is the \textit{mutual coherence}, defined for a frame $\Phi = \{ \phi_i \}_{i=1}^M$ as
\begin{equation} \label{eq:mutcoh}
\mu(\Phi) = \max_{i\neq j} \frac{\vert \langle \phi_i, \phi_j \rangle \vert}{\Vert \phi_i \Vert \Vert \phi_j \Vert}.
\end{equation}
On the one hand, it is clear that when $M=N$ and $\Phi$ is an orthonormal basis, we will have $\mu(\Phi)=0.$ On the other hand, if there exist two linearly dependent vectors in the system, the mutual coherence will be $\mu(\Phi)=1.$ There is a minimal value that $\mu$ can have for a general redundant frame, and it is given by the so-called \textit{Welch bound} \cite{welch1974lower},
\begin{equation} \label{eq:wb}
\mu(\Phi) \geq \sqrt{\frac{M-N}{N(M-1)}}.
\end{equation}
One interesting and very useful property of the Welch bound is that equality in \eqref{eq:wb} is achieved if and only is  $\Phi$ is an equiangular tight frame (ETF) (see \cite[Theorem 2.3]{strohmer2003grassmannian}). These type of frames are in some sense the best (optimally incoherent) redundant frames,  but at the same time highly difficult to construct. Furthermore, the most common guarantees that one is able to reconstruct sparse vectors from linear measurements (to be explained in more details in the last section), 
rely on having small mutual coherence \cite{pfander2008identification}.

We now proceed to the construction of a particular type of frames, which is build from time-frequency shifts of a given  vector. 
\begin{definition}
A \textit{Gabor system} generated by a  window $g \in \C^N$ is the collection
\begin{equation*} 
\Phi_g = \{ M_jT_k g\}_{j,k=0}^{N-1},
\end{equation*}
where $M_jg(n) = e^{\frac{2\pi i jn}{N}} g(n) $ is the modulation (or frequency-shift) operator and $T_k g(n) = g(n-k)$ for all $n=0,\ldots,N-1$ is the translation (or time-shift) operator. We emphasize that all the operations made in the index set $\{0, \ldots, N-1\}$ are in the sense of the group $\Z_N$, that is, modulo $N$.  
\end{definition}
For every $g\neq 0,$ the Gabor system is actually an $N \Vert g \Vert^2$-tight frame \cite{pfander2013gabor}. 
The coherence of this frame, however, depends strongly on the properties of the generator. It is therefore
of interest for many applications to search for ``optimal'' generators.
\end{subsection}

 \begin{subsection}{Definition and coherence properties} 
 We will now be investigating the Gabor frame (and later its generalization to a fusion frame), which is generated by a characteristic function of a difference set, 
 a construction coming from combinatorial design theory \cite{dinitz1992contemporary}, and defined in details below.
 \begin{definition}\label{def:diff-sets}
A subset $\mK = \{u_1,\ldots,u_K\}$ of $\Z_N$ is called an $(N,K,\lambda)$ \textit{difference set}, if the $K(K-1)$ differences
$$(u_k -u_l) \mod N, \quad k \neq l,$$
take all possible nonzero values $1,2,\ldots,N-1,$ with each value exactly $\lambda$ times.
\end{definition}
\begin{example} \label{ex:dfs}
Let $N=7.$ The subset $\mK = \{ 1,2,4\}$ is then a $(7,3,1)$ difference set. We can check this by considering all
possible differences modulo $7$, displayed in the following diagram:
$$ \begin{tabular}{c|ccc}
 - & 1 & 2 & 4  \\ \hline
 1 & - & 6 & 4  \\
 2 & 1 & - & 5  \\
 4 & 3 & 2 & -  
\end{tabular} $$
This confirms that indeed every value from $1$ to $6$ appears exactly one time. For many other examples of difference sets, parameters for which they do not exist and open questions, see the La Jolla Difference Set Repository  \url{http://www.ccrwest.org/ds.html}.
\end{example}
Given  a difference set  $\mK$  with parameters $(N,K,\lambda)$ we denote by $\mChiK \in \C^N$ its
characteristic function:
$$ \mChiK(j) = \begin{cases}
      \hfill 1,    \hfill & \text{if } j \in \mK, \\
      \hfill 0,   \hfill & \text{if } j \notin \mK.\\
      \end{cases}
$$
By $\widehat{\mChiK}$ we denote the Fourier transform of $\mChiK$, which for a general window $g\in \C^N$ is defined as $\widehat{g}(j)=\sum_{k=0}^{N-1}g(k)e^{\frac{-2\pi i kj}{N}}$, for $j=0, \ldots, N-1$. 

We next note some basic, but important properties of difference sets.
\begin{proposition}\label{properties-K}
Let $\mK$ be a difference set with parameters $(N,K,\lambda).$ Then the following is true:
\begin{enumerate}
\item $ K(K-1) = \lambda (N-1)$,  \label{difset-eq}  
\item $\lambda \leq K$, \label{difset-lamK}
\item $\vert \widehat{\mChiK}(j) \vert^2 = K-\lambda, \quad \text{for all } j \in \Z_N \backslash \{0\}$, \label{difset-dft}
\item $\widehat{\mChiK}(0)=K$. \label{difset-dft0} 
\end{enumerate}
\end{proposition}
\begin{proof}
The claim in \ref{difset-eq} comes just from a counting argument: On the one hand, there exist $K(K-1)$ differences in total, and on the other hand, there are $N-1$ numbers that need to appear $\lambda$ times.

Once we have this, for \ref{difset-lamK} we need to check that $\frac{K(K-1)}{N-1} \leq K.$ This inequality is
equivalent to $K(K-1)\leq K(N-1),$ which is always true since $K\leq N.$

Finally, for the Fourier transform, we evaluate
\begin{align*}
\vert \widehat{\mChiK}(j) \vert^2 &= \widehat{\mChiK}(j) \overline{\widehat{\mChiK}(j)} =
                                     \sum_{k, k'\in \mK} e^{\frac{-2\pi i kj}{N}}  e^{\frac{2\pi i k'j}{N}} =
                                     \sum_{k,k'\in \mK} e^{\frac{-2\pi i (k-k')j}{N}} =
                                     \sum_{k \in \mK} 1 +  \sum_{\substack{k,k'\in \mK,\\ k\neq k'}}e^{\frac{-2\pi i (k-k')j}{N}} \\
                                  &= K + \lambda \sum_{\ell=1}^{N-1} e^{\frac{-2\pi i j \ell}{N}} = K +
                                                                  \lambda \left( \sum_{\ell=0}^{N-1} e^{\frac{-2\pi i j \ell}{N}}-1  \right) =
                                      K - \lambda, \text{ when } j\neq 0.
\end{align*}
For $j=0,$ we have $\widehat{\mChiK}(0)=K,$ proving \ref{difset-dft0}.
\end{proof}

Let $\mK$ be a difference set with parameters $(N,K,\lambda)$ and consider the normalized  vector $v := \frac{\mChiK}{\|\mChiK\|}=\frac{\mChiK}{\sqrt{K}} \in \C^N$.
We will denote by $\PhiK$ the Gabor system  generated by $v$,
\begin{equation} \label{eq:def-phi}
\PhiK =\Phi_{v}= \{ M_jT_k v\}_{j,k=0}^{N-1}.
\end{equation}
For short,  we will call $\PhiK$ the Gabor system generated by $\mK$.

Consider the $N \times N^2$ matrix whose columns are the elements of the Gabor system \eqref{eq:def-phi}. We also denote this matrix by $\PhiK$. Further, we write $\PhiK$ as a block matrix,
\begin{equation} \label{eq:phi-block}
 \PhiK = \begin{bmatrix}
           B_0 & B_1 & \ldots & B_{N-1}
          \end{bmatrix},
\end{equation}
where each $B_k$ is a square submatrix of size $N\times N$ with columns of fixed translation, i.e.,
$$B_k = \begin{bmatrix} M_0 T_kv & M_1 T_k v & M_2 T_kv & \ldots & M_{N-1}T_kv \end{bmatrix}.$$

\begin{example}
Let $N=7,$ and let $\omega = e^{2\pi i/7}.$ If we consider the difference set from Example \ref{ex:dfs}, the corresponding matrix $\PhiK,$ with not normalized columns for simplicity, will have the form
$$
\left[
 \begin{array}{ccccccc|ccccccc|c|ccccccc}
  0 & 0       &  0       &     0    &     0       &  0     &    0      & 0 & 0       &  0       &     0    &     0       &  0     &    0  &      &                       1  & 1     &  1  & 1  &  1     & 1 & 1  \\
  1 & \omega  &  \omega^2  & \omega^3  &  \omega^4     & \omega^5 & \omega^6     & 0 & 0       &  0       &     0    &     0       &  0     &    0 &       &             1   & \omega  &  \omega^2  & \omega^3  &  \omega^4     & \omega^5 & \omega^6 \\
  1 & \omega^2 & \omega^4  & \omega^6  &  \omega     & \omega^3 & \omega^5    &	1 & \omega^2 & \omega^4  & \omega^6  &  \omega     & \omega^3 & \omega^5 &       &  0  & 0        &   0    &     0    &     0       &  0     &    0  \\
  0 & 0        &   0    &     0    &     0       &  0     &    0              & 1 &\omega^3 & \omega^6  & \omega^2  &  \omega^5     & \omega & \omega^4 & \, \ldots\,  & 1  & \omega^3  &  \omega^6  & \omega^2  &  \omega^5     & \omega & \omega^4 \\
  1 & \omega^4 & \omega  & \omega^5  &  \omega^2     & \omega^6 & \omega^3    & 0 & 0       &  0       &     0    &     0       &  0     &    0 &        & 0 & 0        &   0    &     0    &     0       &  0     &    0   \\
  0 & 0        &     0   &     0    &     0       &  0     &    0             &	1 & \omega^5 & \omega^3  & \omega  &  \omega^6     & \omega^4 & \omega^2 &       & 0 & 0        &   0    &     0    &     0       &  0     &    0  \\
  0 & 0        &    0    &     0    &     0       &  0     &    0       &	0 & 0       &  0       &     0    &     0       &  0     &    0 &        & 0 & 0        &   0    &     0    &     0       &  0     &    0  
 \end{array}
 \right].
$$ 

\end{example}

Going back to the general construction $\PhiK$ and the investigation of its coherence, we note that the Gram matrix of $\PhiK,$ which is defined as $G = \PhiK^*\PhiK,$ is closely related to the mutual coherence. Namely,
\begin{equation}\label{eq:coherence-gram}
 \mu(\Phi) = \max_{i \neq j} \left| G(i,j) \right|.
\end{equation}

For our Gabor system, using the notation from \eqref{eq:phi-block}, the Gram matrix can be written in the block form
\begin{equation}
 G = \begin{bmatrix}
      B_0^* \\ B_1^* \\ \cdots \\ B_{N-1}^*
     \end{bmatrix}
     \begin{bmatrix}
      B_0 & B_1 & \cdots & B_{N-1}
     \end{bmatrix}
      =\begin{bmatrix}
        B_0^*B_0 & B_0^*B_1 & \cdots & B_0^*B_{N-1} \\
        B_1^*B_0 & B_1^*B_1 & \cdots & B_1^*B_{N-1} \\
                 &          &  \cdots &              \\
        B_{N-1}^*B_0 & B_{N-1}^*B_1 & \cdots & B_{N-1}^*B_{N-1}
       \end{bmatrix}.
\end{equation}
We will next state a property of the diagonal blocks in $G,$ which will later turn out to be useful.
 \begin{proposition}\label{prop-BoBo}
Under the notation given above, we have that 
$$\vert B_k^*B_k(j,\ell) \vert =  \begin{cases}
  \sqrt{\frac{N-K}{K(N-1)}},& \textrm{ if } j\neq \ell,\\
 1,& \textrm{ if } j=\ell,
 \end{cases} $$
for all  $\, k,j, \ell=0, \ldots, N-1$.
In particular, the diagonal blocks $B_0^*B_0, \, B_1^*B_1,\, \ldots, \, B_{N-1}^*B_{N-1}$ are all equal in absolute value. 
\end{proposition}
 \begin{proof}
We will first prove that any entry of the blocks $B_k^*B_k,$ $k=1,\ldots,N-1$ is equal in absolute value to the corresponding one in the first block $B_0^*B_0.$
Let $k$ be some element from $\{1,2,\ldots,N-1\}.$ Using the definition of $B_k$ and the basic properties of translation and modulation operators, we have
$$\vert B_k^*B_k(j,\ell) \vert =    \vert  \langle  M_\ell T_k v, M_j T_k  v \rangle \vert  = \vert  \langle e^{\frac{-2\pi i k\ell}{N}}T_k M_\ell v, e^{\frac{-2\pi i kj}{N}} T_k M_j v \rangle \vert  = \vert \langle M_\ell v, M_j v \rangle \vert =\vert B_0^*B_0(j,\ell) \vert,$$
for all $j,\ell =0,1,\ldots,N-1.$
Now, according to the definition of $B_0$ and Proposition \ref{properties-K} \ref{difset-dft}-\ref{difset-dft0},
\begin{equation*}
 \vert B_0^*B_0(j,\ell) \vert = \frac{1}{K} \left| \sum_{k\in \mK} e^{\frac{2\pi i (\ell-j)k}{N}} \right| = \frac{1}{K} \left| \widehat{\mChiK} (j-\ell) \right| = 
 \begin{cases}
 \frac{1}{K} \sqrt{K-\lambda} = \sqrt{\frac{N-K}{K(N-1)}}& \textrm{ if } j\neq \ell,\\
 1& \textrm{ if } j=\ell.
 \end{cases}
\end{equation*}
\end{proof}
\begin{remark} \label{rem:known}
Note that, for each $k=0, \ldots, N-1$,  Proposition \ref{prop-BoBo} says that the collection of $N$ vectors $\{M_0 T_kv,  M_1 T_k v,  \ldots, M_{N-1}T_kv\}$ which spans a $K$-dimensional subspace of $\C^N$, has coherence achieving the Welch bound. Therefore, by \cite[Theorem 2.3]{strohmer2003grassmannian},  $\{M_0 T_kv,  M_1 T_k v,  \ldots, M_{N-1}T_kv\}$ is an ETF for the subspace it spans, for  every  $k=0, \ldots, N-1$ and it has frame bound $\frac{N}{K}.$ 
Note, that this result was proven in \cite[Theorem 1]{xia2005achieving}, where equiangular tight frames are called maximum-Welch-bound-equality (MWBE) codebooks.
It is unclear at this point, however, what the absolute values of the entries in the off diagonal blocks are. As we will see in the next theorem, they will depend on the value of $\lambda,$ and thus the mutual coherence of $\PhiK$ will depend on the 
parameters of the difference set $\mK.$
\end{remark}
\begin{theorem} \label{th:mu} Let $\PhiK$ be a Gabor system generated by an  $(N,K,\lambda)$ difference set $\mK.$ Then,
 \begin{align*}
  \mu(\PhiK) &=  \begin{cases}
      \hfill \sqrt{\frac{N-K}{K(N-1)}},    \hfill & \text{if } \lambda=1,\\
      \hfill \max \{\frac{K-1}{N-1}, \sqrt{\frac{N-K}{K(N-1)}}\},   \hfill & \text{if } \,\lambda \neq 1.
      \end{cases}\\
  \end{align*}
 \end{theorem}

\begin{proof}
According to the described block structure of $\PhiK$ and by \eqref{eq:coherence-gram},  the mutual coherence is
\begin{align*}
\mu(\PhiK) &= 
                \max \{ \max_{\substack{r\neq q \\ j,\ell}} \vert B_r^*B_q(j,\ell) \vert,\,  \max_{\substack{j\neq \ell}} \vert B_0^*B_0(j,\ell)\vert \}.
\end{align*}
We have already investigated the diagonal blocks in Proposition \ref{prop-BoBo}. Next we write explicitly the elements of the Gram matrix $G$ in the off-diagonal blocks as

\begin{align} \label{eq:brbq}
      \vert B_r^*B_q(j,\ell) \vert &= \vert \langle M_\ell T_q v, M_j T_r v \rangle \vert = \vert \langle M_\ell v, T_{(r-q)}M_j v \rangle \vert \nonumber \\
       &=\left| \sum_{k=0}^{N-1} v(k) e^{\frac{2\pi i k \ell}{N}}\overline{ v(k-(r-q))} e^{\frac{-2\pi i (k-(r-q))j}{N}}  \right|
      =  \frac{1}{K} \left| \sum_{\substack{k \in \mK\\ k-(r-q)\in \mK}} e^{\frac{2\pi i (k \ell+(r-q)j-kj)}{N}}\right| \nonumber \\
      &=\frac{1}{K} \left|  \sum_{\substack{k \in \mK\\ k-(r-q)\in \mK}} e^{\frac{2\pi i (\ell-j)k}{N}}\right| .
\end{align}
We can simplify this expression further dependent on the properties of the difference set $\mK.$ We thus consider two separate cases.

\textit{Case $\lambda=1$.}
In the final sum \eqref{eq:brbq}, in the case $q \neq r,$ since $\lambda=1,$ there can be only one $k\in \mK,$ such that $k$ and $k-(r-q)$ are both in $\mK$. This is because there is only one way to write $r-q$ as a difference of elements in $\mK,$ and $k-(k-(r-q))$ is such a difference. Thus, we can continue \eqref{eq:brbq} to obtain
\begin{equation*}
 \vert B_r^*B_q(j,\ell) \vert = \frac{1}{K} \left| e^{\frac{2\pi i (\ell-j)k}{N}} \right| = \frac{1}{K}, \quad \text{when } q\neq r.
\end{equation*}
Further,  by Proposition \ref{prop-BoBo}, $\vert B_0^*B_0(j,\ell) \vert =  \sqrt{\frac{N-K}{K(N-1)}}$. Therefore, when $\lambda=1,$ 
$\mu(\Phi)=\max\{\frac{1}{K}, \sqrt{\frac{N-K}{K(N-1)}}\}=\sqrt{\frac{N-K}{K(N-1)}}.$

\textit{Case $\lambda\neq 1$.}
We will estimate $\max_{r\neq q, j,\ell} \vert B_r^*B_q(j,\ell) \vert$. For fixed $r\neq q$, since $\mK$ is a $(N,K, \lambda)$ difference set we have that $\{k \in \mK: k-(r-q)\in \mK\}$ is a set of exactly $\lambda$ elements. Then, 
from \eqref{eq:brbq} it follows that for all $j,l=0, \ldots, N-1,$
$$\vert B_r^*B_q(j,\ell) \vert\leq \frac{\lambda}{K}.$$
Note that when $j=\ell$, also by \eqref{eq:brbq},  $\vert B_r^*B_q(j,j) \vert =\frac{\lambda}{K}$.  Thus $\max_{r\neq q, j,\ell} \vert B_r^*B_q(j,\ell) \vert=\frac{\lambda}{K}$. Now we just use the fact that $K(K-1)=\lambda(N-1)$ to rewrite $\frac{\lambda}{K}$ as $\frac{K-1}{N-1}.$
\end{proof} 
 
\begin{remark}
Although the value $\sqrt{\frac{N-K}{K(N-1)}}$ was optimal for the case of $N$ vectors in $K$ dimensional space, 
for the full Gabor frame $\PhiK$ the optimal Welch bound will be different. Namely, for a system of $N^2$ vectors in $N$ dimensional 
space, the Welch bound is
$$\mu^* = \sqrt{\frac{N^2-N}{N(N^2-1)}} = \sqrt{\frac{1}{N+1}}.$$
Below we present a table with several families of difference sets and the mutual coherence of the corresponding Gabor systems. For more details on the construction of these difference sets see \cite{xia2005achieving}.
\renewcommand{\arraystretch}{3}
$$
\begin{array}{|l|c|c|c|}
\hline
 \text{Family} & (N,K,\lambda) & \mu(\PhiK)^2 & \mu^{*2}  \\ \hline  
 \text{Singer, $d=2$} & \left(q^2+q+1, q+1, 1 \right) & \dfrac{q}{(q+1)^2}&  \dfrac{1}{q^2+q+2} \\ \hline
 \text{Singer, $d>2$} & \left(\dfrac{q^{d+1}-1}{q-1}, \dfrac{q^d-1}{q-1}, \dfrac{q^{d-1}-1}{q-1}\right) & \dfrac{(q^d-q)^2}{q^2(q^d-1)^2}& \dfrac{q-1}{q^{d+1}+q-2}  \\ \hline
 \text{Quadratic, $q >7 $} & \left(q, \dfrac{q-1}{2}, \dfrac{q-3}{4}\right)  & \dfrac{(q-3)^2}{4(q-1)^2} &  \dfrac{1}{q+1} \\ \hline
 \text{Quartic, $p<57$} & \left(p, \dfrac{p-1}{4}, \dfrac{p+3}{16}\right)  & \dfrac{3p+1}{(p-1)^2} & \dfrac{1}{p+1} \\ \hline
 \text{Quartic, $p>57$} & \left(p, \dfrac{p-1}{4}, \dfrac{p+3}{16}\right)  &  \dfrac{(p-5)^2}{16(p-1)^2} & \dfrac{1}{p+1} \\ \hline
\end{array}
$$
\renewcommand{\arraystretch}{1}
From this table it can be seen that the mutual coherence is not as close to the optimal bound, as it was established, for example, for the Alltop vectors in  \cite{strohmer2003grassmannian}. It is still going asymptotically to zero as the dimension grows for the Singer family, and, as we will see in the numerical experiments, the performance of the difference sets and the Alltop vectors for the sparse recovery problem are almost identical, 
making our construction still interesting for applications.

Reaching the Welch bound is important not only for signal processing, but actually for many other fields, including quantum mechanics, where ETFs of $N^2$ elements in dimension $N$ are known as SIC-POVMs (symmetric informationally complete positive-operator valued measure). It is in fact an open problem whether they exist for every dimension (Zauner conjecture \cite{zauner1999quantendesigns}). Particular examples are also difficult to construct, and known only for certain values of $N.$

From our reasoning above, one might conclude that it is possible to get a Gabor ETF by choosing a difference set with optimal values of the parameters $K$ and $N,$ such that $\mu(\Phi) = \mu^*.$ Constructing a difference set with prescribed parameters is however itself a very difficult and open problem in combinatorial design theory. It is also directly connected to the optimal Grassmannian packing problem \cite{conway1996packing}. Up to now only constructions with certain pairs of parameters $(N,K)$ are known. In any case, we will show that, unfortunately, combinations of parameters of difference sets such that corresponding Gabor system achieves the Welch bound can not exist. Such hope was  probably too good to be true, since for illustration, for $N=17$ an analytical example of a generator which gives an ETF of $N^2$ lines in dimension $N$ was provided in \cite{chien2015equiangular}, but it took over 40 pages to write its expression down.
\end{remark}

\begin{proposition}
Let $N>3.$ Then, there can not exist an  $(N, K, \lambda)$ difference set such that the corresponding Gabor system $\PhiK$ 
will form an equiangular tight frame.
\end{proposition}
\begin{proof}
By Theorem \ref{th:mu}, the mutual coherence can take only one of the two possible values:
$$\sqrt{\frac{N-K}{K(N-1)}} \text{ or } \frac{K-1}{N-1}.$$
We will now consider these two cases separately. 

Let us first assume that $\mu(\Phi) = \sqrt{\frac{N-K}{K(N-1)}}.$ If we want to reach the Welch bound, we need to solve $\frac{N-K}{K(N-1)}=\frac{1}{N+1},$ which implies $K=\frac{N+1}{2}.$ From Proposition \ref{properties-K}, we know that the corresponding $\lambda$ in this case is $\frac{N+1}{4}.$ 
However, for this set of parameters $\left(N,\frac{N+1}{2}, \frac{N+1}{4}\right)$, when $N>3,$ it is easy to check that the mutual coherence is actually $\max \left\{ \sqrt{\frac{N-K}{K(N-1)}}, \frac{K-1}{N-1} \right\}  
= \frac{K-1}{N-1} = \frac{1}{2},$ and thus far from the Welch bound $\mu^* =  \sqrt{\frac{1}{N+1}}$. 
It is interesting to note that, when $N=3,$ potential difference sets with parameters $(3,2,1)$ will achieve the Welch bound. An example for such a difference set is $\mK = \left\{ 0,1 \right\}.$ Its characteristic function $\tilde{g}=\begin{bmatrix} 1 & 1 & 0 \end{bmatrix}$ (which is a member of a continuous family of generators presented in \cite{chien2015equiangular}) forms a Gabor frame of $9$ elements which is an ETF.

Let us next see what happens if we want to reach the Welch bound with the other value, i.e. to have $\left(\frac{K-1}{N-1}\right)^2 = \frac{1}{N+1}.$ For positive $K$ this equation is solved by $K = \frac{N+1+\sqrt{N^3-N^2-N+1}}{N+1}.$ But for such $K$ and $N>3,$ the mutual coherence will actually be $\sqrt{\frac{N-K}{K(N-1)}}$ instead of $\frac{K-1}{N-1}$, and thus again we can not reach the Welch bound. Note that when $N=3,$ $K$ in the obtained solution is again $2.$ Actually,   $\sqrt{\frac{N-K}{K(N-1)}} = \frac{K-1}{N-1} = \frac{1}{\sqrt{N+1}} = \frac{1}{2},$ and thus we again achieve the Welch bound.
\end{proof}

\end{subsection}

\end{section}

\begin{section}{Fusion frames coming from difference sets} \label{sec:ff}
We now move to the second part of the paper, where we aim to investigate our collection of time-frequency shifts of a difference set from a perspective of fusion frames, which are collections of subspaces and generalize the notion of frames. Constructing fusion frames with prescribed ``frame-like'' properties is an important and challenging task. We will show how our Gabor system can be seen as a fusion frame, 
and that it moreover satisfies certain optimality properties which will be discussed further. We now recall the definition of fusion frames \cite{casazza2008fusion}, in our case considered with all weights equal to one.
\begin{definition}
A family of subspaces $\{ \mW_i \}_{i=1}^M$ in  $\C^N$ is called
a \textit{fusion frame} for $\C^N,$ if there exist $A$ and $B,$ $0 < A \leq B < \infty$ such that
\begin{equation*} \label{eq:ff}
 A \Vert x \Vert_2^2 \leq \sum_{i=1}^M \Vert P_i(x) \Vert_2^2 \leq B \Vert x \Vert_2^2 \quad \text{ for all } x \in \C^N,
\end{equation*}
where for each $i=0, \ldots, N-1$, $P_i$ denotes the orthogonal projection of $\C^N$ onto $\mW_i$.
\end{definition}
If $A=B$ is possible, then $\{ \mW_i \}_{i=1}^M$ is called an \textit{$A$-tight fusion frame}. 
Tightness is an important property, required for example, for minimization of the recovery error of a random vector from its noisy fusion frame measurements \cite{kutyniok2009robust}. Among other desirable properties are equidimensionality and equidistance. They provide maximal robustness against erasures of one or more subspaces, and as we will see later, yield optimal Grassmannian packings \cite{kutyniok2009robust}. Equidimensionality means that all the subspaces $\{ \mW_i \}_{i=1}^M$ are of the same dimension, while   to define equidistant fusion frames, we need the notion of chordal distance. 

\begin{definition}
Let $\mW_1$ and $\mW_2$ be subspaces of $\C^N$ with $m:=\dim \mW_1 = \dim \mW_2$ and denote by $P_i$ the orthogonal 
projection onto $\mW_i, \, i=1,2.$ 
The \textit{chordal distance} $d_c(\mW_1,\mW_2)$ between $\mW_1$ and $\mW_2$ is given by
\begin{equation*}
 d_c^2(\mW_1, \mW_2) = m - \Tr[P_1 P_2],
\end{equation*}
where $\Tr$ denotes the trace of an operator. Multiple subspaces are called \textit{equidistant}, if they have pairwise equal chordal distance $d_c.$
 \end{definition}

It was shown in \cite{kutyniok2009robust} that equidistant tight fusion frames are optimal Grassmannian packings, where optimality comes from the classical packing problem:
For given $m,M,N,$ find a set of $m$-dimensional subspaces $\{ \mW_i \}_{i=1}^M$ in $\C^N$ such that $\min_{i\neq j} d_c(\mW_i,\mW_j)$ is as large as possible. In this case
we call $\{ \mW_i \}_{i=1}^M$ an \textit{optimal packing}. An upper bound is given by the \textit{simplex bound}
\begin{equation*} \label{eq:simplexb}
 \frac{m(N-m)M}{N(M-1)}.
\end{equation*}
This is to some extent analogous to the Welch bound from the classical frame theory, and we will 
see that fusion frames generated by difference sets actually achieve the simplex bound.

We will investigate the family of subspaces arising from Gabor system of difference sets, defined as follows.
Let $\mK$ be a difference set with parameters $(N,K,\lambda)$ and let $v =\frac{1}{\sqrt{K}} \mChiK$ be our generator 
for the Gabor system
$$\PhiK = \{  M_j T_i v \}_{j,i=0}^{N-1}.$$
For every $i=0,\ldots,N-1,$ let the subspaces $\mW_i$ be defined as 
\begin{equation} \label{eq:ff-sets}
\mW_i = \Span \{  M_j T_i v \}_{j=0}^{N-1} = \{ x \in \C^N : \supp(x) =  \mK+i \}.
\end{equation}
We call $\mWK = \{ \mW_i \}_{i=1}^N$  \textit{a Gabor fusion frame associated to a difference set} $\mK.$
The fact that this family of subspaces is in fact a fusion frame (and more over tight) will follow from the next proposition.

\begin{proposition}[Corollary 13.2 in \cite{casazza2012finite}] \label{pr:cor132}
Let $\{ \mW_i \}_{i=1}^M$ be a family of subspaces in $\C^N.$ Let $\{\phi_{ij}\}_{j=1}^{J_i}$ be an $A$-tight frame for $\mW_i$ for each $i.$ Then the
following conditions are equivalent.
\begin{enumerate}
 \item $\{ \mW_i \}_{i=1}^M$ is a $C$-tight fusion frame for $\C^N.$  
 \item $\{\phi_{ij}\}_{i=1,j=1}^{M,J_i}$ is an $AC$-tight frame for $\C^N.$
\end{enumerate}
\end{proposition}

\begin{theorem} \label{th:tightff}
 The family of subspaces $\mWK = \{ \mW_k \}_{k=0}^{N-1}$ defined in \eqref{eq:ff-sets} is a $K$-tight fusion frame.
\end{theorem}

\begin{proof}
This property follows directly from Proposition \ref{pr:cor132}. First of all, as noted in Remark \ref{rem:known}, for every fixed $i,$ $\{ M_j T_i v\}_{j=0}^{N-1}$ is a $\frac{N}{K}$-tight (also equiangular) frame for $\mW_i.$ Also, the full system $\{ M_j T_iv\}_{i,j=0}^{N-1}$ is a $N$-tight frame for $\C^N.$  Thus, according to Proposition \ref{pr:cor132}, this is equivalent to $\{ \mW_i \}_{i=0}^{N-1}$ being a $K$-tight fusion frame for $\C^N.$
\end{proof}

\begin{subsection}{Equidistant fusion frames}
We saw that our construction produces tight fusion frames consisting of equi-dimensional subspaces. Next, we will show that they moreover have equal pairwise chordal distance.

\begin{theorem} \label{th:equi-ff}
The Gabor fusion frame $\mWK = \{ \mW_k \}_{k=0}^{N-1}$ associated to an  $(N,K,\lambda)$ difference set $\mK$ is an equidistant fusion frame  with $$d_c^2 =\frac{K(N-K)}{N-1}.$$
\end{theorem}

\begin{proof}
Let $\mW_{i_1}$ and $\mW_{i_2}$ be any two different subspaces from \eqref{eq:ff-sets}. 
To compute $d_c^2(\mW_{i_1},\mW_{i_2})$ we require $\Tr[P_{i_1} P_{i_2}] = \sum_{\ell=0}^{N-1} \langle P_{i_2} e_\ell, P_{i_1} e_\ell \rangle,$ where $\{e_\ell\}_{\ell=0}^{N-1}$ is the canonical basis of $\C^N$. For this, first note that 
$$ P_{i_k} \,e_\ell = \frac{K}{N} \sum_{j=0}^{N-1} \langle e_\ell, T_{i_k} M_j v \rangle T_{i_k} M_j v =  \frac{K}{N} \sum_{j=0}^{N-1} \overline{T_{i_k} M_j v(\ell)}\, T_{i_k} M_j v, \quad k=1,2.$$
This leads to
\begin{align} \label{eq:distsum}
\sum_{\ell=0}^{N-1} \langle P_{i_2} e_\ell, P_{i_1} e_\ell \rangle 
         &=  \frac{K^2}{N^2} \sum_{\ell=0}^{N-1} \langle \sum_{j=0}^{N-1} \overline{ M_j T_{i_2} v(\ell)}\,  M_j T_{i_2} v, \sum_{j'=0}^{N-1} \overline{ M_{j'}T_{i_1}v(\ell)}\, M_{j'} T_{i_1} v \rangle \nonumber \\
         &=  \frac{K^2}{N^2} \sum_{\ell=0}^{N-1} \sum_{j=0}^{N-1} \sum_{j'=0}^{N-1}\overline{ M_jT_{i_2} v(\ell)} M_{j'}T_{i_1} v(\ell) \langle  M_j T_{i_2} v,  M_{j'}T_{i_1} v \rangle \nonumber \\
         & = \frac{K^2}{N^2} \sum_{j,j'=0}^{N-1} \langle M_{j'} T_{i_1}  v, M_{j}  T_{i_2} v \rangle \langle  M_{j}T_{i_2} v,  M_{j'} T_{i_1} v \rangle 
          = \frac{K^2}{N^2} \sum_{j,j'=0}^{N-1} \vert  \langle  M_jT_{i_1} v,  M_{j'} T_{i_2}v \rangle \vert^2 \nonumber \\
        &  \overset{\eqref{eq:brbq}}{=} \frac{1}{N^2} \sum_{j,j'=0}^{N-1} \left| \sum_{\substack{k \in \mK\\ k-(i_2-i_1)\in \mK}} e^{\frac{2\pi i (j'-j)k}{N}}  \right|^2
          = \frac{1}{N^2} \sum_{j,j'=0}^{N-1} \vert \hat{\chi}_{\mK_{i_1-i_2}}(j-j') \vert^2,
\end{align}
where $\mK_{i_1-i_2} = \{ k \in \mK : k-(i_1-i_2) \in \mK\}.$ As we have noted before, $\card(\mK_{i_1-i_2}) = \lambda.$ For fixed $j,$ by Plancherel's Theorem, we have 
$$\sum_{j'=0}^{N-1} \vert \hat{\chi}_{\mK_{i_1-i_2}}(j-j') \vert^2 = \Vert T_j \hat{\chi}_{\mK_{i_1-i_2}} \Vert^2 =  \Vert \hat{\chi}_{\mK_{i_1-i_2}} \Vert^2 
  = N \Vert \chi_{\mK_{i_1-i_2}} \Vert^2 = N \lambda.$$
Now we can go back to the sum \eqref{eq:distsum}, and get the final result,
\begin{equation*}
 \sum_{\ell=0}^{N-1} \langle P_{i_2} e_\ell, P_{i_1} e_\ell \rangle  =   \frac{1}{N^2} \sum_{j=0}^{N-1}  N \lambda = \lambda.
\end{equation*}
Notice that this value does not depend on the choice of the subspaces. Thus, taking into account that our subspaces have dimension $K,$ by definition of chordal distance we obtain
$$d_c^2 = K - \Tr[P_1 P_2] = K - \lambda.$$
Finally, by Proposition \ref{properties-K} \ref{difset-lamK}, the claim follows.
\end{proof}

\begin{corollary}
The Gabor fusion frame $\mWK = \{ \mW_k \}_{k=0}^{N-1}$ associated to an  $(N,K,\lambda)$ difference set $\mK$ is an optimal Grassmannian packing of $N$ $K$-dimensional subspaces in $\C^N.$
\end{corollary}
\begin{proof}
By Theorem \ref{th:equi-ff}, $\mWK$ is a fusion frame of equidimensional subspaces with pairwise equal chordal distances $d_c.$ It was proven in \cite[Theorem 4.3]{kutyniok2009robust}, that in this case the fusion frame is tight, if and only if $d_c^2$ equals the simplex bound. We already know from Theorem \ref{th:tightff} that our fusion frame is tight, hence the claim follows. We can also check that the simplex bound is achieved. For this set of parameters, the simplex bound equals
$$ \frac{K(N-K)N}{N(N-1)} = \frac{K(N-K)}{N-1},$$
and this is exactly $d_c^2.$ Thus, we have an optimal packing.
\end{proof}

\end{subsection}

\begin{subsection}{Optimally sparse fusion frames}
The notion of optimally sparse fusion frames was introduced in \cite{casazza2011optimally} and means that all subspaces can be seen as spans of orthonormal basis vectors that are sparse in a uniform basis over all subspaces, and thus only few entries are present in the decomposition. 
This  different optimality property is of great practical use when the fusion frame dimensions are large, and low-complexity fusion frame decomposition is desirable.
We will show that our Gabor fusion frames defined in \eqref{eq:ff-sets} are also optimally sparse.

\begin{definition} \cite{casazza2011optimally} \label{def:sparseff} Let $\{ \mW_i \}_{i=1}^M$ be a fusion frame for $\C^N$ with $\dim\mW_i=m_i$ for all
$i=1,\ldots,M$ and let $\{v_j\}_{j=1}^N$ be an orthonormal basis for $\C^N.$ If for each $i\in \{1,\ldots,M\},$ there
exists an orthonormal basis $\{\phi_{i,\ell}\}_{\ell=1}^{m_i}$ for $\mW_i$ with the property that for each $\ell=1,\ldots,m_i$
there exists a subset $J_{i,\ell} \subset \{1,\ldots,N\}$ such that
$$ \phi_{i,\ell} \in \Span \{ v_j : j \in J_{i,\ell} \} \, \text{ and } \, \sum_{i=1}^M \sum_{\ell=1}^{m_i} \vert J_{i,\ell} \vert = k,$$
we refer to $\{ \phi_{i,\ell}\}_{i=1,\ell=1}^{M,m_i}$ as an \textit{associated $k$-sparse frame.} The fusion frame
$\{ \mW_i \}_{i=1}^M$ is called \textit{$k$-sparse} with respect to $\{v_j\}_{j=1}^N,$ if it has an associated $k$-sparse frame
and if, for any associated $j$-sparse frame, we have $k \leq j.$
\end{definition}

\begin{definition} \cite{casazza2011optimally}
Let $\mF\mF$ be a class of fusion frames for $\C^N,$ let $\{ \mW_i \}_{i=1}^M \in \mF\mF,$  and let $\{v_j\}_{j=1}^N$ 
be an orthonormal basis for $\C^N.$ Then $\{ \mW_i \}_{i=1}^M$ is called \textit{optimally sparse in $\mF\mF$ with
respect to} $\{v_j\}_{j=1}^N,$ if $\{ \mW_i \}_{i=1}^M$ is $k_1$-sparse with respect to $\{v_j\}_{j=1}^N$ and there does
not exist a fusion frame $\{\mV_i\}_{i=1}^M \in \mF\mF$ which is $k_2$-sparse with respect to $\{v_j\}_{j=1}^N$ with $k_2 < k_1.$
\end{definition}
Let $\mF\mF(M,m,N)$ be the class of tight fusion frames in $\C^N$ which have $M$ subspaces, each of dimension $m.$
One example of optimally sparse fusion frames in this class is the spectral tetris construction (STFF), explained in more details
in \cite{casazza2011optimally} and \cite[Chapter~13]{casazza2012finite}. For this fusion frame the following theorem is known.

\begin{theorem}  \cite{casazza2011optimally} \label{thm:tetris}
Let $N,M,$ and $m$ be positive integers such that $\frac{Mm}{N} \geq 2$ and $\lfloor \frac{Mm}{N} \rfloor \leq M-3.$
 Then the tight fusion frame STFF $(M,m,N)$ is optimally sparse in the class $\mF\mF(M,m,N)$ with respect 
 to the canonical basis in $\C^N$. 
 
In particular, this tight fusion frame is $mM+2(N-\gcd(Mm,N))$-sparse with respect to the canonical basis.
\end{theorem}

We will now show that the Gabor fusion frames generated by difference sets are also optimally sparse in the corresponding class of tight fusion frames.
\begin{theorem}
Let $\mWK=\{ \mW_k \}_{k=0}^{N-1}$ be the Gabor fusion frame associated with a difference set $\mK$  with parameters $(N,K,\lambda).$ 
Then, $\mWK$ is an optimally sparse fusion frame in the class $\mF\mF(N,K,N)$ with respect to the canonical basis with sparsity $KN.$
\end{theorem}

\begin{proof}
From Theorem \ref{th:tightff}, we know that $\mWK$ is a tight fusion frame from the class $\mF\mF(N,K,N).$

Let $\{ e_j \}_{j=1}^N$ be the canonical basis of $\C^N$. From the definition of $\mWK$ \eqref{eq:ff-sets} it follows that the elements of each subspace $\mW_i$ are supported on the sets $\mK+i.$ Therefore, as an orthonormal basis for every $\mW_i$  we can take
$$\{ \phi_{i,\ell} \}_{\ell=1}^K, \text{ where } \phi_{i,\ell} = e_{k_\ell+i}, k_\ell \in \mK.$$
Then, the corresponding sets $J_{i,\ell}$ from Definition \ref{def:sparseff} are each of cardinality $1,$ and the sparsity of $\mWK$ is
$$ \sum_{i=1}^N \sum_{\ell=1}^{K} \vert J_{i,\ell} \vert = KN.$$
Now, for any other associated $j$-sparse frame with sets $\{\tilde{J}_{i,\ell}\}_{i=1, \ell=1}^{N, K}$, we have that $ \sum_{i=1}^N \sum_{\ell=1}^{K} \vert \tilde{J}_{i,\ell} \vert\geq KN$ because each $ \tilde{J}_{i,\ell}$ has at least one element. Thus, $\mWK$ is $KN$-sparse. Moreover, this also says that $KN$ is the smallest sparsity that one can expect in  $\mF\mF(N,K,N)$. Therefore, $\mWK$ is optimally sparse.  
\end{proof}

\begin{remark}
 For $K \geq 2,$ $K \leq N-3,$ we have by Theorem \ref{thm:tetris} that STFF $(N,K,N)$ is optimally sparse in $\mF\mF(N,K,N)$. Note that in this case the sparsity given by Theorem \ref{thm:tetris}, $KN+2(N-\gcd(NK,N))$ is exactly $KN$.
\end{remark}

\end{subsection}

\end{section}

 \begin{section}{Numerical experiments} \label{sec:num}
In this section we test our two constructions based on difference sets: Gabor frames and Gabor fusion frames, on recovery of sparse signals from given Gabor-like measurements. 
\subsection{Classical sparse recovery} 
We aim to recover an unknown sparse (having small number of nonzero entries) vector $x \in \C^{N^2}$ from its linear measurements $y=\Phi_g x,$ where $\Phi_g$ is the $N\times N^2$ Gabor system generated by $g \in \C^N$. 
Recovery of sparse signals from linear measurements is the classical compressed sensing setup \cite{candes2006stable}. To recover $x$ we will traditionally use Basis Pursuit (BP) \cite{chen1998atomic}, which is the convex problem given by
\begin{equation*} \label{eq:l1}
 \min \Vert x \Vert_1 \text{ subject to } \Phi_g\, x = y.
\end{equation*}
We solve this problem with CVX, a package for specifying and solving convex programs \cite{cvx}.
We want to compare the results of recovery of sparse vectors using three different types of generators for $\Phi_g:$ Alltop vectors \cite{alltop1980complex}, complex random vectors and difference sets:
\begin{itemize}
 \item[$g_A:$] Alltop vector. $g_A(j) = \frac{1}{\sqrt{N}} e^{2\pi i j^3 / N},$ for prime $N \geq 5.$
 \item[$g_R:$] Random vector. $g_R(j) = \frac{1}{\sqrt{N}} \epsilon_j,$ where $\epsilon_j$ are independent and uniformly distributed on the torus.
 \item[$g_K:$] Difference set. $g_K = \frac{1}{\sqrt{K}} \mChiK$ for some $(N,K,\lambda)$ difference set $\mK.$
 \end{itemize}
We have chosen Alltop and random generators, since their Gabor frames have already proven to be successful for sparse recovery both theoretically and numerically \cite{pfander2008identification}. The theoretical guarantees come from the near optimality of the mutual coherence of these Gabor systems. We would like to see how the difference sets compare to their performance, despite their theoretically non-optimal coherence.
 
In the numerical experiment in Figure \ref{fig:1}, we have chosen $N=43$ (a prime which gives $3$ modulo $4,$ suitable for difference sets of the Quadratic family that we will use). For fixed sparsity level $k$, we generate a random $k$-sparse vector $x \in \C^{N^2},$ with $k$ non-zero values $x(j) = r_j \exp(2\pi i \theta_j),$ where $r_j$ is drawn independently from the standard normal distribution $\mN(0,1),$ and $\theta_j$ is drawn independently and uniformly from $[0,1).$ Then, we measure this signal with each of the three Gabor frames, and try to recover it by BP. We count the recovery as successful, if the normalized squared error was smaller than $10^{-6}.$ 
For every $k$ we repeat this experiment $T=500$ times, and plot the successful recovery rates in Figure \ref{fig:1}. 
What we observe here is that all three generators have almost identical recovery rate. The complex random generator performs the best since we choose a different realization for $g_R$ at every experiment, and the difference sets are slightly better then the Alltop in the transition level of $k.$ 

The fact that the mutual coherence can not always capture the desired properties of the Gabor frame was noted in \cite{bajwa2010gabor}, where \textit{average coherence} was introduced. To guarantee a successful recovery via BP, certain relations between the average and the mutual coherence need to be satisfied. One can show that those particular conditions are also not satisfied by the Gabor frame generated by difference sets. Finding the correct theoretical explanation of this  successful behavior in numerical experiments is an interesting question, and we leave it for future investigation.

\begin{figure}[h!] 
\centering
 \includegraphics[scale=0.48]{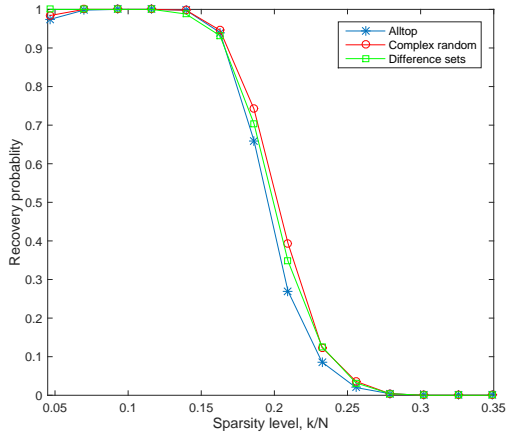}  
 \caption{Sparse recovery from Gabor measurements} \label{fig:1}
\end{figure}

\subsection{Fusion frame sparse recovery}
Here the task is to use Gabor fusion frame $\mWK$ generated by a difference set $\mK$ for recovery of signals which are sparse in a fusion frame, namely, which have nonzero components lying in only few subspaces. In our case that would correspond to having only few translations. A detailed theoretical description of this problem for fusion frames in general, and its importance for applications is given in \cite{boufounos2011sparse}. Investigations on recovery from random fusion frames were conducted in \cite{ayaz2013sparse}.

The problem we are given in short is as follows. Given the condition that $x = \{x_j\}_{j=1}^N, \, x_j \in \mW_j$ has only few nonzero components $x_j,$ recover $x$ from its measurements $y = A_P x,$ where $A_P = \{ a_{ij} P_j \}_{i,j=1}^{n,N}.$ Here, $\{ a_{ij} \}_{i,j=1}^{n,N}$ is the measurement matrix which in our experiment we take to be random Gaussian, and $P_j$ are the projections onto the corresponding subspaces $\mW_j.$
This time $x$ can be found by solving the minimization problem
\begin{equation} \label{eq:ffmin}
\min_{x \in \mH} \Vert x \Vert_{2,1}  \text{ subject to } A_P x = y.
\end{equation}
The norm which we minimize is the mixed $\ell_2 \slash \ell_1$ norm, which promotes ``block'' like sparsity, and is defined as
$$ \Vert x \Vert_{2,1} = \sum_{j=1}^N \Vert x_j \Vert_2, \text{ where }  x = \{x_j\}_{j=1}^N, \, x_j \in \mW_j.$$
Although it might not be clear how to numerically solve the described minimization process \eqref{eq:ffmin}, this can be easily accomplished via the standard $\ell_1$ minimization technique, by incorporating the basis vectors for each of the subspaces. The details can be found in \cite{boufounos2011sparse}.
Here the question is not only what level of sparsity are we able to recover, but also, how many measurements $n$ in the matrix $A_P$  we need. At the same time, as we will see, the dimensions of the subspaces will play an important role.

In Figure \ref{fig:K13}, we consider an  $(N,K,\lambda)$ difference set with $N=40$ and $K=13,$ which means that the dimension of the subspaces is $13,$ and in the experiment depicted in Figure \ref{fig:K21}, we set $N=43$ and $K=21.$ For a different number of measurements, as denoted in the legend, and for every sparsity level, we generate random $k$-fusion sparse vector  $x$ 
(with independent random Gaussian values at $k$ subspaces chosen at random). Then, we calculate the measurements $y,$ and try to recover back $x$ by \eqref{eq:ffmin} again using CVX. We repeat each experiment $T=100$ times, and count the recovery as successful, if the normalized squared error was smaller than $10^{-6}$. The results are presented in Figure \ref{fig:ff}. We observe that as expected, larger number of measurements allows for higher levels of sparsity, but also that when the dimension of the subspaces is smaller, 
fewer measurements are needed to recover the signal. Moreover, if the subspaces are of small dimension, and the number of measurements is sufficiently large, we can recover $x$ independently of its sparsity level.

The theoretical results for this problem, described in \cite{boufounos2011sparse}, are again not sufficient to capture this effect: leaving out the details, the \textit{fusion coherence}, which is defined as 
$$ \mu_f (A_P, \{ \mW_i \}_{i=1}^M )= \max_{j \neq k} \left[ \vert \langle a_j, a_k \rangle \vert \cdot \Vert P_j P_k \Vert_2 \right]$$
and guarantees successful recovery, in our case will be equal to the mutual coherence of the Gaussian measurement matrix, and will not depend on the fusion frame structure, since one can prove that $\Vert P_j P_k \Vert_2$ always equals $1.$ Therefore, a more subtle measure of coherence is still missing in both problems presented, and these questions will be part of future research.

\begin{figure}[h!] 
\centering
\begin{subfigure}{0.49\linewidth} \centering
 \includegraphics[scale=0.48]{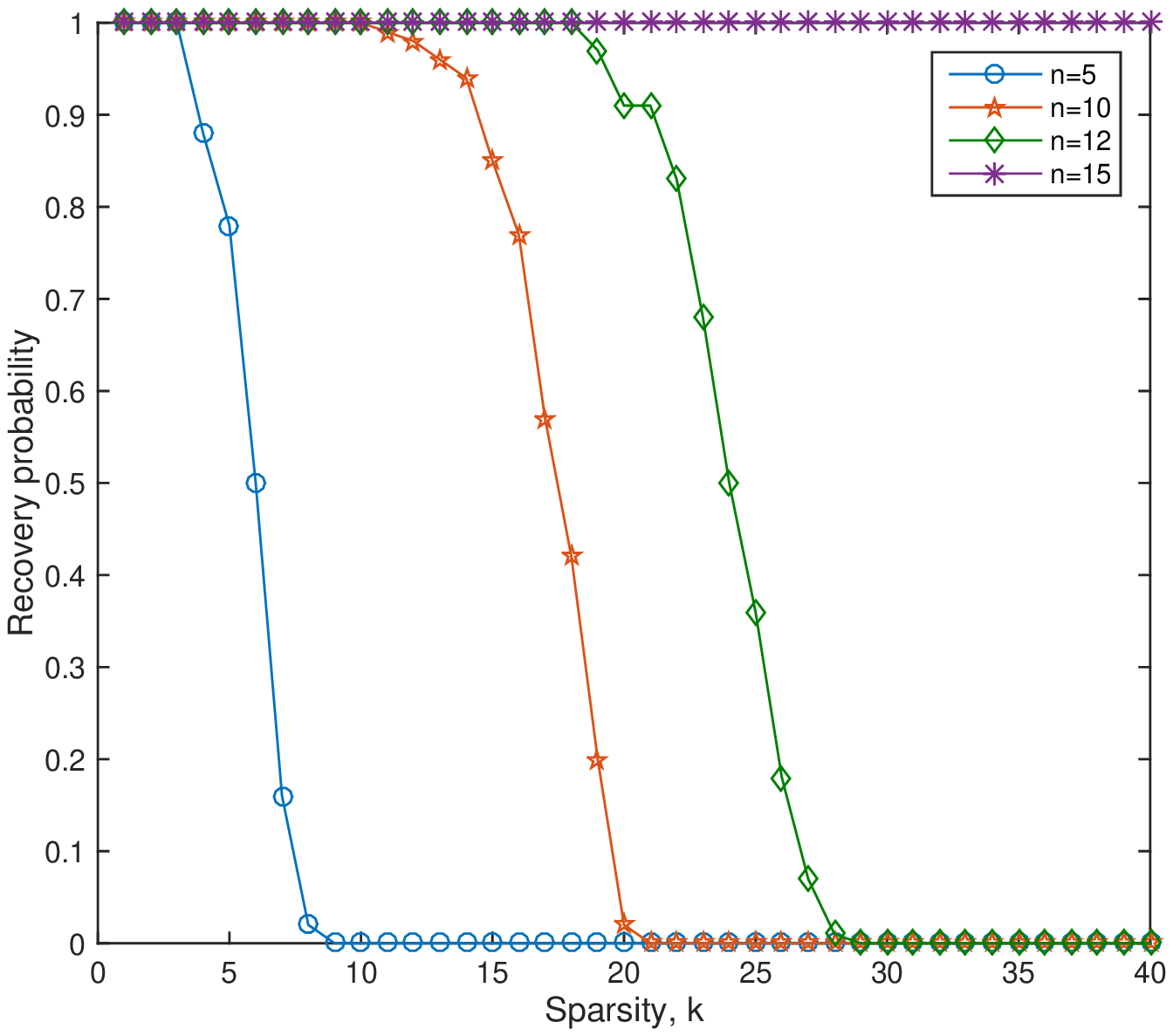}  
 \caption{Difference set with $N=40, K=13$} \label{fig:K13}
 \end{subfigure}
 \begin{subfigure}{0.49\linewidth} \centering
 \includegraphics[scale=0.48]{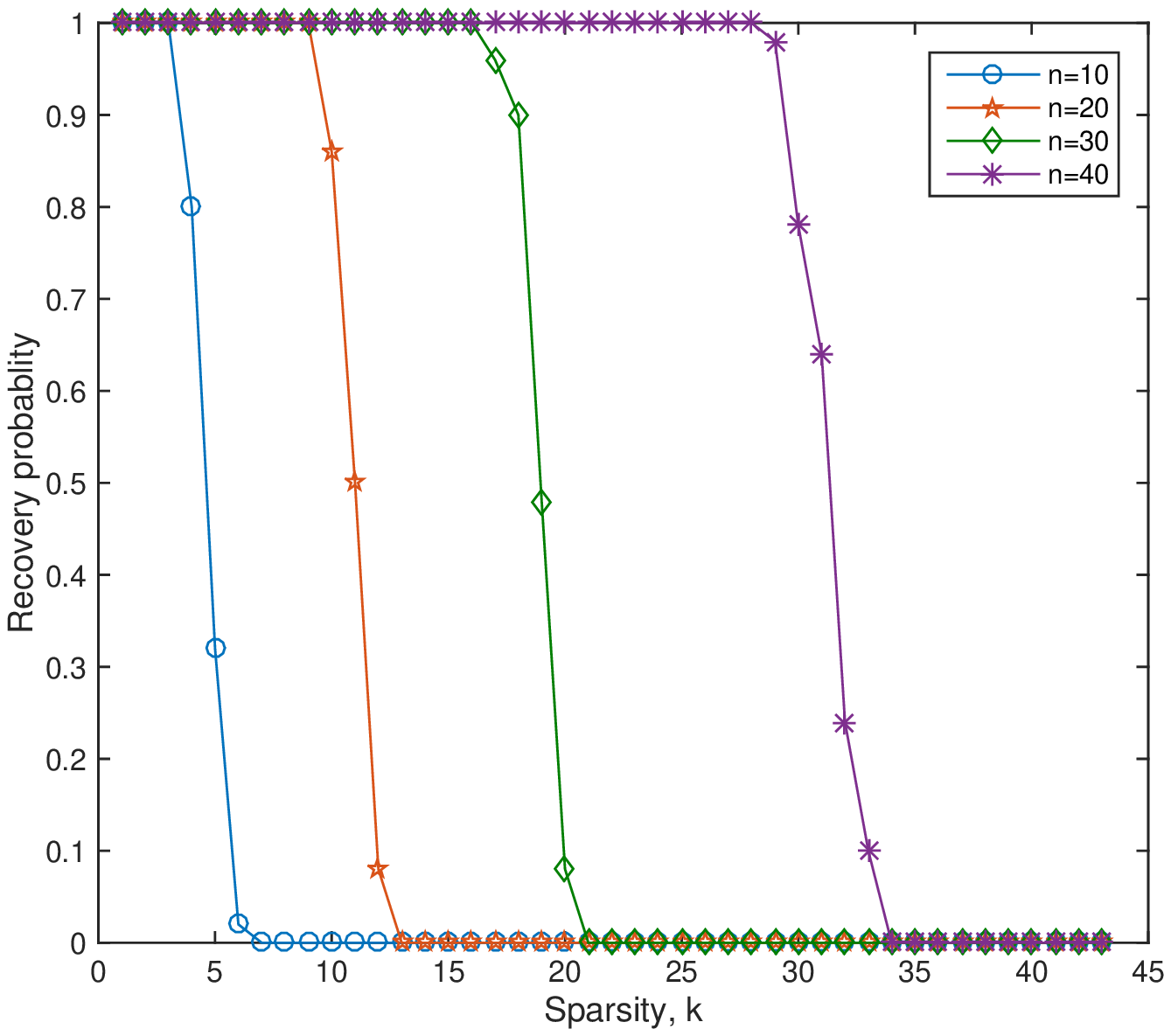}
 \caption{Difference set with $N=43, K=21$}  \label{fig:K21}
 \end{subfigure}
 \caption{Fusion sparse recovery with Gabor fusion measurements} \label{fig:ff}
\end{figure}

 \end{section}

\section*{Acknowledgment}
The authors would like to thank Gitta Kutyniok and Matthew Fickus for fruitful discussions and useful comments that greatly improved the manuscript. 
I.~B. acknowledges support by the Berlin Mathematical School. V.~P. was supported by a fellowship for postdoctoral researchers
from the Alexander von Humboldt Foundation and by Grants UBACyT  2002013010022BA and CONICET-PIP 11220110101018.

\end{document}